\documentclass{cimart}

\usepackage{commath}

\title{Normalized solutions for Schr\"odinger-Bopp-Podolsky systems in bounded domains}

\author{Gaetano Siciliano}

\authorinfo{Universit\`a degli Studi di Bari Aldo Moro, Bari, Italy}
{gaetano.siciliano@uniba.it}

\abstract{%
    We consider an elliptic system of Schr\"odinger-Bopp-Podolsky type  in a bounded and smooth  domain of $\mathbb R^3$  with a non-constant coupling factor. This type of system has been introduced in the mathematical literature in \cite{AveniaSiciliano2019} and in the last years many contributions appeared. In particular here we present the results in \cite{AS} and \cite{LG} which show existence of solutions, under different boundary conditions on the electrostatic potential. The problem has indeed a variational formulation and by means of the Critical Point Theory  and the Lusternik-Schnirelmann theory, we prove that the associated energy functional has infinitely many critical points, which then give solutions of the problem.
}

\keywords{Schr\"odinger-Bopp-Podolsky system, Krasnoselskii genus, Lagrange multipliers, weak solutions.}

\msc{35J50 (primary), 35J58, 35Q55, 35Q61 (secondary)}

\VOLUME{34}
\YEAR{2026}
\ISSUE{2}
\NUMBER{3}
\DOI{https://doi.org/10.46298/cm.15501}

\begin{document}

\section{Introduction}

This paper is a survey presenting results obtained in collaboration with the Brazilian mathematicians D. G. Afonso and L. Soriano Hernández (see \cite{AS} and \cite{LG}) on the existence and multiplicity of solutions of a nonlinear elliptic system arising in the physical sciences. More specifically, we consider a system arising from the coupling, in the context of the Abelian Gauge Theories, of the electromagnetic field in the formulation given by Bopp and Podolsky with the Schr\"odinger equation, whose main reference is \cite{AveniaSiciliano2019} (see also the references therein). For classical references about Quantum Mechanics, charged particles and Gauge theory, we refer the interested reader to the classical books \cite{F,N}.

The first paper about this kind of coupling dates back to 1998 where in \cite{BenciFortunato1998} the authors consider the Schr\"odinger equation and the electromagnetic field in the classical formulation of Maxwell. This line of investigation has attracted then many researchers worldwide and in the last decades many interesting papers appeared on this matter, involving also quasilinear and fractional equations. We  also made contributions with some Brazilian mathematicians and the interested reader is referred to
\cite{BS1,BS2, CT,GaeLilPie, EdwinG, EdwinG2,FS2,FS3, FMS,LPT,SS2} and the references therein.

In this paper, we focus on the electromagnetic theory developed by Bopp and Podolsky. Let us  start by presenting briefly some of the reasons which led Bopp and Podolsky to consider a new electromagnetic theory of second order, which, in many cases, is preferable to the classical Maxwell one.

The Bopp-Podolsky theory was developed by Bopp and Podolsky independently (see \cite{B40,Pob42}) to overcome the so called {\em infinity problem}  of the Maxwell theory, which consists in the following. It is known,  by the  Gauss law in Physics (or Poisson equation in Mathematics), that  the electrostatic potential $\phi$ generated by a charge distribution with  density $\rho$ satisfies the equation
\begin{equation}\label{Gauss}
	-\Delta\phi=\rho
	\qquad \text{in }\mathbb R^3.
\end{equation}
In particular if $\rho=\delta_{x_0}$, the Dirac delta function centered in  $x_0\in\mathbb R^3$, the solution  is $\mathcal{G}(x-x_0)$, where
\[
    \mathcal{G}(x)=\frac{1}{4\pi |x|},
\]
whose electrostatic energy is
\[
    \mathcal{E}_{\rm M}(\mathcal{G})
    =\frac{1}{2}\int_{\mathbb R^3} |\nabla \mathcal{G}|^2 =+\infty.
\]
In the Bopp-Podolsky theory the Poisson equation \eqref{Gauss} for the electrostatic potential  is replaced by
\begin{equation}\label{GaussBP}
    -\Delta \phi + a^2 \Delta^2 \phi = \rho
    \qquad
    \hbox{in }\mathbb R^3
\end{equation}
where $a\neq0$ is a parameter and $\Delta^{2}$ is the bi-Laplace operator, namely $\Delta^{2}\phi := \Delta(\Delta \phi)$. In case of a charge particle $\rho=\delta_{x_0}$, the solution which is physically admissible is
\[
    \mathcal{K}(x)
    :=\frac{1-e^{-|x|/a}}{4\pi |x|}, \quad a>0.
\]
Note that it is not singular in $x_0$ since
\[
    \lim_{x\to x_0} \mathcal{K}(x-x_0)=\frac{1}{4\pi a},
\]
and its energy is finite, since, as one easily computes it holds
\[
    \mathcal{E}_{\rm BP}(\mathcal{K})
    =\frac{1}{2}\int_{\mathbb R^3} |\nabla \mathcal{K}|^2
    +\frac{a^{2}}{2} \int_{\mathbb R^3} |\Delta \mathcal{K}|^2
    <+\infty.
\]
In practice, in this theory the energy of a particle is finite and its potential finite everywhere. Moreover for large distances, namely $|x-x_{0}|>>1$, it is $\mathcal G(x-x_{0})\approx \mathcal K(x-x_{0})$ so the two theories are experimentally indistinguishable.

Observe that the electrostatic equation \eqref{GaussBP} is still in divergence form since it can be written as
$$ \nabla \cdot \mathbb V = \rho$$
with $\mathbb V = \nabla (a^{2}\Delta \phi - \phi)$. Moreover classical methods permits to find that the fundamental solution is exactly the difference between the fundamental solution of the Poisson equation $\mathcal G$ (the Coulomb potential) and the fundamental solution of the Helmholtz equation $\mathcal K$ (the Yukawa potential).

Indeed it is a general fact. If $L$ is a linear differential operator, and in the sense of distributions it is $Lf=\delta$ and $Lg+kg=\delta$, for $k\neq0$, then we have
\begin{eqnarray*}
    \frac{1}{k} L(L+k)(f-g)
    &=&    \frac1k L(Lf+kf - Lg-kg) \\
    &=& \frac{1}{k} L(\delta + kf - \delta)\\
    &=& \delta.
\end{eqnarray*}

We are not interested in the rigorous deduction of the equations we are going to study, but of course the obtaining of such equations passes through a modification of the Maxwell Lagrangian which is used in this new electromagnetic theory. In fact Bopp and Podolsky deal with the following Lagrangian
\begin{align*}
    \mathcal{L}_{\rm BP}
    &=
    \frac{1}{8\pi}\left\{
    |\nabla\phi+\frac{1}{c}\partial_t\mathbf{A}|^2  -|\nabla \times \mathbf A|^{2}
    \right.\\
    &\qquad\qquad
    \left.+ a^2 \left[ \left(\Delta \phi+\frac{1}{c}\operatorname{div} \partial_t \mathbf{A}\right)^2 - \left|\nabla\times\nabla\times\mathbf{A} + \frac{1}{c}
    \partial_t(\nabla\phi+\frac{1}{c}\partial_t\mathbf{A})\right|^2\right]
    \right\},
\end{align*}
instead of the classical
$$\mathcal{L}_{\rm M}=
\frac{1}{8\pi}\left\{
|\nabla\phi+\frac{1}{c}\partial_t\mathbf{A}|^2 -|\nabla \times \mathbf A|^{2}
\right\},$$
to which it reduces when $a=0$. In the expressions above $\mathbf A$ and $\phi$ are the gauge potentials of the electromagnetic field.

The Euler-Lagrange equations of $\mathcal L_{\textrm{BP}}$ lead to a modification of the well known Maxwell equations.
For physical and technical details we refer the reader to the papers \cite{BPVZ,BPO,BPS2014,BPS2017,CDMPP,CDMNP}.

\subsection{Statement of the results}

In this paper we consider a very particular situation of the Euler-Lagrange equations with the energy functional derived from the Lagrangian $\mathcal L_{\textrm{BP}}$, namely of the coupling between the matter field (governed by the nonlinear Schr\"odinger equation) and the electromagnetic field in the Bopp-Podolsky theory. In fact in the purely electrostatic case (i.e. independence on time and null magnetic field) looking at standing waves solutions
$\psi(x,t) = u(x)e^{i\omega t}$ in presence of a suitable nonlinearity, a system of the following type appears:
\begin{equation*}
    \begin{cases}
    - \Delta u + q(x) \phi u  -  |u|^{p-2}u   = \omega u\\
    - \Delta \phi+a^{2}\Delta^2 \phi   = q(x) u^2.
    \end{cases}
\end{equation*}
The deduction of the equations can be found in \cite{AveniaSiciliano2019}.

We will study the system above in a bounded and smooth domain $\Omega \subset \mathbb R^3$. For simplicity we consider the case with $a=1$; namely the system of equations under study is
\begin{equation}\label{eq:P1}
    \begin{cases}
    - \Delta u + q(x) \phi u  -  |u|^{p-2}u   = \omega u\\
    - \Delta \phi+\Delta^2 \phi   = q(x) u^2.
    \end{cases}
\end{equation}
Here $q\in C(\overline\Omega)\setminus\{0\}$ is a given datum and represents a non-uniform charge distribution, $p\in(2,10/3)$ and the unknowns are
\[
	u,\phi:\Omega\to \mathbb R \quad \text{ and }\quad  \omega\in \mathbb R.
\]

Let us make a brief discussion on the unknowns and the conditions they have to satisfy in order to guarantee the existence of solutions.

About $u$, we assume
\begin{equation}\label{eq:b1}
    u = 0 \quad \text{ on } \ \partial \Omega
\end{equation}
and
\begin{equation} \label{eq:normalization_condition}
    \int_{\Omega} u^2  dx = 1.
\end{equation}

The homogeneous Dirichlet boundary condition simply reminds the fact that the particle is constrained to live in $\Omega$: in fact it ``disappears'' if touching the boundary. The second condition is justified since $u$ represents the amplitude of the wave function of the particle confined in $\Omega$. It is then a quite reasonable normalizing condition which has a physical meaning being related to the probability of finding the particle in the region $\Omega$. It also has interesting mathematical consequences; indeed, it is a conserved quantity along the solution of the evolution problem and it appears naturally in the eigenvalue problem.

For what concerns the unknown $\phi$, the electrostatic potential, we consider two different situations.

\textbf{(I)} We start with the case of a boundary condition on the potential and its Laplacian
\begin{align}
    \phi & = 0\quad \text{ on } \ \partial \Omega, \label{eq:b5} \\
    \Delta \phi& = 0\quad \text{ on } \ \partial \Omega, \label{eq:b4}
\end{align}
namely, homogeneous Dirichlet boundary conditions (with some abuse of notation, as in contrast to the boundary conditions in the next case). Indeed these are called Navier boundary condition. Here the reason for which we have chosen homogeneous boundary conditions is just technical. We could consider non-homogeneous boundary conditions as well. The usual strategy when dealing with non-homogenous boundary conditions is to define a suitable auxiliary problem, which incorporates the boundary conditions, and then make a change of variable in order to reduce the problem to an homogeneous one. In fact this is what we will do for the problem with the Neumann boundary conditions in \textbf{(II)}. Then to avoid technical details, we consider in case \textbf{(I)} homogeneous boundary conditions.

\textbf{(II)} We consider the case of a boundary condition on the normal derivative of the potential and its Laplacian
\begin{align}
    \partial_{\mathbf n}{\phi} & = h_1 \quad \text{ on } \ \partial \Omega, \label{eq:b2} \\
    \partial_{\mathbf n}{\Delta \phi} & = h_2 \quad \text{ on } \ \partial \Omega, \label{eq:b3}
\end{align}
where for simplicity $h_1, h_2$ are continuous functions defined on $\partial\Omega$.

The symbol $\mathbf n$ denotes the unit vector normal to $\partial \Omega$ pointing outwards, so we are in the presence of Neumann boundary conditions on the potential and its Laplacian. These boundary conditions are related to the flux of the electric field through the boundary. Observe that in case $h_{1}=h_{2}=0$, by integrating the second equation in \eqref{eq:P1} in $\Omega$, we find
\[
	0= \int_{\Omega} q(x) u^{2}dx.
\]
In particular the function $q$ has to change sign in order to find solutions (in particular it cannot be constant). However, as we will see in Theorem \ref{thm:main_theorem} we need a further condition which involves the function $q$.

We point out that having the real number $\omega$ as an unknown of the problem (joint with $u$) means that the wave function is completely unknown, no a priori frequency is given. We remark that it is also possible to consider the problem where $\omega$ is a priori given: the approach is a little bit different; in fact solutions can be found also in this case but there is no a priori control on the $L^{2}$ norm. However we believe it is more natural from a physical point of view to look also at the frequency as unknown. This is related with the normalizing condition since from a variational point of view, $\omega$ is just the Lagrange multiplier of the energy functional restricted to the sphere and related to its critical point which is exactly the solution $u$. In some sense the value of the frequency and the $L^{2}-$norm are dual one to the other.

Furthermore, the range $p\in(2,10/3)$ guarantees that the energy functional of the problem is bounded below on the manifold made by the functions with fixed $L^{2}-$norm. This is the so-called \textit{mass or $L^{2}-$subcritical case}, and an interesting paper concerning the behaviour of functionals of this type on such a constraint is \cite{SS2}, where the role of the exponent $10/3$ is evident.

Finally observe that the normalizing value of the $L^{2}-$norm is not essential: the results are true for any other positive value different from $1$, without any relation with the size of~$\Omega$.

Here are the main results we are able to obtain by using variational methods and tools in critical point theory. However, in order to state them precisely, we introduce the functional spaces where the solutions will be found.

As usual, $H^{2}(\Omega)$ and $H^{1}_{0}(\Omega)$ denote the Sobolev spaces. In particular $H^{1}_{0}(\Omega)$ is endowed with the scalar product
\[
	(u,v) = \int_{\Omega} \nabla u \nabla v dx
\]
and associated norm
\[
	 \left\|u\right\|:=\left(\int_{\Omega}\left|\nabla u\right|^{2}dx\right)^{{1}/{2}}.
\]
Let $\mathbb{H}:=H^{1}_{0}(\Omega)\cap H^{2}\left(\Omega\right)$, whose norm (denoted for simplicity with the same symbol) is
\[
	\left\|\phi\right\|:=\left(\int_{\Omega}\left|\Delta \phi\right|^{2}dx+\int_{\Omega}\left|\nabla \phi\right|^{2}dx\right)^{{1}/{2}},
\]
which of course comes from the scalar product (again denoted with the same symbol as before)
\[
	(\phi,\psi)=\int_{\Omega}\Delta \phi\Delta\psi dx+\int_{\Omega}\nabla \phi\nabla\psi dx.
\]
Throughout this work we denote by
\[
	|u|_{p}:=\left(\int_{\Omega}\left|u\right|^{p}dx\right)^{{1}/{p}}
\]
the norm in $L^{p}(\Omega)$. Usually we will not write the variable $x\in \Omega$ in the functions involved; except for the function $q$ to emphasize the fact that $q$ is not a constant, in contrast to many papers dealing with this type of system.

We define the $L^{2}-$sphere in $H^{1}_{0}(\Omega)$ by
\[
	B:=\left\{u\in H^{1}_{0}(\Omega): |u|_{2}=1\right\},
\]
which, in view of the previous discussion about the unknown $\omega$ seen as the Lagrange multiplier, will have a major role.

Here is the first result, taken from \cite{LG}.

\begin{theorem}[{Existence result with boundary conditions \textbf{(I)}}] \label{tma1.1}
    Let $p \in (2, 10/3)$. Then there exists a sequence of solution $\{(u_{n},\omega_{n},\phi_{n} )\} _{n}\subset B\times \mathbb R\times \mathbb H $ of problem \eqref{eq:P1} under conditions \eqref{eq:b1}-\eqref{eq:normalization_condition} and \eqref{eq:b5}-\eqref{eq:b4}, with
    $$
    \omega_{n} \to+ \infty, \quad \|u_{n}\| \to +\infty, \quad \text{ as } \ \  n\to+\infty.
    $$
\end{theorem}

The second result is from \cite{AS}.

\begin{theorem}[{Existence result with boundary conditions \textbf{(II)}}]
\label{thm:main_theorem}
    Let $p \in (2, 10/3)$ and set
    \begin{equation}\label{eq:alfa}
    \alpha := \int_{\partial \Omega} h_2 ds - \int_{\partial \Omega} h_1 ds.
    \end{equation}
    Assume that
    \[
    	\inf_\Omega q < \alpha < \sup_\Omega q\quad \text{ and } \quad |q^{-1}(\alpha)| = 0.
    \]
    Then there exist infinitely many solutions $\{(u_n, \omega_n, \phi_n)\} \subset B\times \mathbb R \times H^2(\Omega)$ to the problem \eqref{eq:P1} under conditions \eqref{eq:b1}-\eqref{eq:normalization_condition} and \eqref{eq:b2}-\eqref{eq:b3}, with
    $$\int_{\Omega}q(x)u_{n}^{2}dx=\alpha,\quad \omega_{n}-\frac{\alpha}{|\Omega|}\int_{\Omega}\phi_{n} dx\to+\infty, \quad \|u_{n}\|\to+\infty,
    \quad \text{ as } \ \  n\to+\infty.$$
\end{theorem}

Our approach is variational, indeed the solutions will be found as critical points of an energy functional restricted to a suitable constraint. In this context, by a ground state solution we mean the solution with minimal energy. As a byproduct of the proof, we obtain, for both results, that also the energy of these solutions is divergent and that there exists always a ground state solution that can be assumed positive.

We conclude by referring the reader to \cite{AveniaSiciliano2019,FS,H,MS,RS,HS,SS}, and the references therein, for some recent results concerning the existence of solutions for systems of type \eqref{eq:P1}. We cite finally the papers \cite{CRZ,MSJDE, PV, NTV, RQSS,S} about Schr\"odinger type equations on bounded domains and prescribed $L^{2}-$norm.

The paper is organized as follows. We will treat the two problems separately in Section~\ref{sec:Dir}, for Dirichlet boundary conditions, and in Section \ref{sec:aux}, for Neumann boundary conditions.

For each of them we will introduce the appropriate variational setting in order to address the problem. In particular an energy functional has to be defined with the appropriate geometric and compactness properties. The solution will be found as critical points of the energy functional on a suitable smooth manifold, whose differential and topological properties are shown in order to apply the critical point theory and variational methods, more specifically the Lusternik-Schnirelmann theory. However, in case of Dirichlet boundary conditions the manifold we will consider is the $L^{2}$ sphere, which is evidently a (nonempty) manifold. In contrast, for the Neumann boundary conditions we will restrict the functional to a subset of the $L^{2}$ sphere and we need to show it is nonempty and possesses a differentiable structure of a manifold. To this aim the assumptions on $\alpha$ will have a fundamental role. The problem with Neumann boundary conditions is also more involved since these conditions are non homogeneous, so a suitable auxiliary problem has to be introduced.

However, we begin with the next Section \ref{sec:prelim}, where we present some useful preliminaries.

As a matter of notations, we use the letters $c, c',\ldots $ to denote positive constant whose value can change from line to line.

\section{Few general facts}
\label{sec:prelim}

We spend a few words on our methods. We use Critical Point Theory in order to show how the solutions can be associated to a critical point of a functional on a suitable manifold in a Hilbert space.
In view of the applications of variational methods and to use topological invariants of the Lusternik-Schnirelmann Theory, some facts like compactness and geometry of the functional have to be shown.

We recall that in many problems of this type, the frequency $\omega$ of the wave function is fixed.
Then the approach in finding solutions is different, in particular the $L^{2}$ norm of the solutions $u$
is not given a priori. In our case, the wave function is completely unknown, so both $u$ and $\omega$ are unknowns, and we are looking for solutions with a priori fixed $L^{2}$ norm. Let us recall that the $L^{2}$ norm is constant in time on the solutions of the evolution problem, so it is constantly equal to the $L^{2}$ norm of the initial datum. As a consequence, the unknowns $\omega$ related to the solutions will be found as the Lagrange multipliers associated to the critical points on the manifold made by the unit sphere in $L^{2}$, or some submanifold. For these reasons, we think that it is natural to consider the frequencies of the wave function, $\omega$, as an unknown and the $L^{2}$ norm of $u$ fixed, since it is more interesting also from a physical point of view.

Let us start by recalling the definition of genus of Krasnoselskii. Given $A$ a closed and symmetric subset of some Banach space, with $0\notin A$, the \textit{genus }
of $A$ is denoted as $\gamma(A)$ and defined as the least integer $k$ for which there exists a continuous and odd map $h:A\to \mathbb R^{k}\setminus\{0\}$. By definition it is $\gamma(\emptyset) = 0$ and if it is not possible to construct continuous odd maps from $A$ to any $\mathbb R^{k}\setminus\{0\}$, set $\gamma(A) = +\infty$.

We know that the genus is a topological invariant (by odd homeomorphism) and that the genus of the sphere in $\mathbb R^{m}$ is $m$ (see e.g., \cite{Rab}).

Recalling that $B$ is the $L^{2}-$sphere in $H^{1}_{0}(\Omega)$, we have the following.

\begin{lemma}\label{lemma7}
    For any integer $m$ there exists a compact and symmetric subset $\textsc{K}$ of $B$
    such that $\gamma(\textsc{K})=m$. In other words, the genus of $B$ is infinite.
\end{lemma}

\begin{proof}
    Let $\textsc{H}_{m}:=span\left\{u_{1},\ldots,u_{m}\right\}$ be a $m$-dimensional subspace of $H^{1}_{0}\left(\Omega\right)$. Define
    \[
    	\textsc{K}:=B\cap \textsc{H}_{m}=\left\{u\in \textsc{H}_{m}: |u|_{2}=1\right\}.
    \]
    Consider the odd homeomorphism \begin{equation*}\begin{array}{rcl}
    h:\textsc{K}&\longrightarrow& \mathbb{S}^{m-1}\\
    u&\longmapsto&h(u)= \displaystyle\frac{x}{\left\|x\right\|_{\mathbb{R}^{m}}},
    \end{array} \text{ where } x=(x_{1},\dots, x_{m})\in \mathbb{R}^{m}.
    \end{equation*} By the genus invariance via odd homeomorphism (see e.g., \cite[Proposition 5.4]{Struwe}), we obtain
     \[
    	\gamma(\textsc{K})=\gamma(\mathbb{S}^{m-1})=m
    \]
     concluding the proof.
\end{proof}

A useful tool in Critical Point Theory to obtain the compactness is the well known Palais-Smale condition that we recall now. We say that a $C^{1}$ functional $J $ on a Hilbert space $H$, satisfies the Palais-Smale condition on the differentiable manifold $\mathcal M\subset H$
if any sequence $\{w_{n}\}\subset \mathcal M$ such that
\[
	\{J (w_{n})\}\ \text{ is bounded and }\ J' (w_{n}) \to 0 \ \text{ in }\ \textrm{T}_{w_{n}}\mathcal M,
\]
called also a Palais-Smale sequence, has a convergent subsequence in the $H$ norm. Of course $\textrm{T}_{w_{n}}\mathcal M$ is the tangent space in $w_{n}$ to the manifold $\mathcal M$.

Let us recall now a classical result in critical point theory. We give the proof for the reader convenience, see also \cite{BenciFortunato1998}.

\begin{lemma}
\label{lem:sublevels_have_finite_genus}
    Let $J$ be a $C^{1}$ functional defined on a smooth manifold $\mathcal M$ in a Hilbert space $H$ which is bounded from below and satisfies the Palais-Smale condition.
    For any $b \in \mathbb R$ the sublevel
    $$
    J^b = \left\{u \in \mathcal M \ : \ J(u) \leq b\right\}
    $$
    has finite genus.
\end{lemma}

\begin{proof}
    Suppose by contradiction that there exists a real number $c$ such that $\gamma(J ^{c})=+\infty$; this means that
     $$
     D:=\left\{b\in \mathbb{R}: \gamma(J ^{b})=\infty\right\}\neq \emptyset.
     $$
     We know that $J $ is bounded from below on $\mathcal M$, hence
     \begin{equation*}
     -\infty<{b^{*}}:=\inf D < +\infty.
     \end{equation*}
      We claim that $b^{*}\notin D$. Indeed, since $J $ satisfies the Palais-Smale condition on $\mathcal M$ the set
     \[
    	K_{b^{*}}:=\left\{u\in \mathcal M: J (u)=b^{*}, \  \ J'(u)=0\right\}
    \] is compact.
    By the property of the genus (see e.g. \cite[Proposition 5.4]{Struwe}), there exists a closed symmetric neighborhood $Z$ of $K_{b^{*}}$ such that $\gamma(Z)<+\infty$, then $b^{*}\notin D$.
    
    By the deformation lemma (see e.g. \cite[Theorem 3.11]{Struwe}), there exist $\varepsilon >0$ and an odd homeomorphism $\eta$ such that $\eta(1,J ^{b^{*}+\varepsilon}\setminus Z)\subset J ^{b^{*}-\varepsilon}$.
    Using properties (2), (3) and (5) of \cite[Proposition 5.4]{Struwe}, we get
    $$\gamma(J ^{b^{*}+\varepsilon})\leq \gamma(J ^{b^{*}+\varepsilon}\setminus Z)+\gamma(Z)\leq \gamma(J ^{b^{*}-\varepsilon})+\gamma(Z)<+\infty,$$
    which goes against the fact that $b^{*}$ is equals to $\inf D$. Therefore for all $c\in\mathbb{R}$ it has to be $\gamma(J ^{c})<+\infty$.
\end{proof}

The solutions will be obtained as critical points of a suitable energy functional. Rigorously this gives rise to what are called \textit{weak solutions}, in the sense that just integral (and not point-wise) equalities are satisfied. However, standard boot-strap arguments show that they are classical and that the equations are satisfied point-wise in $\Omega$. Let us see the proof for the solutions given in Theorem \ref{tma1.1}. The same happens for the solutions in Theorem~\ref{thm:main_theorem}.

We use the classical notation for H\"older spaces  $C^{j,\lambda}(\overline \Omega)$, $0 < \lambda \leq 1$.

Let $(\omega,u,\phi) \in \mathbb R\times B\times \mathbb H$ be a weak solution of \eqref{eq:P1}, then $\psi:=-a^{2}\Delta \phi +\phi$ is a weak solution of the Dirichlet problem
\begin{equation*}
     \left\{\begin{array}{rcll}
    -\Delta \psi  & = &q(x)u^{2}&\text{ in } \Omega,\\
    \psi & = & 0 &\text{ on }\partial \Omega.\end{array}\right.
\end{equation*}

Now, if $u\in H^{1}_{0}\left(\Omega\right)$, then $u\in L^{6}\left(\Omega\right)$ and  $u^{2}$ belongs to $L^{3}\left(\Omega\right)$. Thus, by \cite[Theorem~9.9]{Gilbarg1988} we have 
\begin{equation}
    -a^2 \Delta \phi + \phi = \psi \in W^{2,3}\left(\Omega\right).\label{A.1}
\end{equation}
Recall that $\Omega$ is a bounded set. If $\phi \in \mathbb{H}$ is a solution of \eqref{A.1} with $\psi \in W^{2,2}\left(\Omega\right)$, the interior regularity increases because  \cite[Theorem 8.10]{Gilbarg1988} implies that $\phi\in W^{4,2}\left(\Omega\right)$ and so $\phi \in C^{2,\lambda}(\overline{\Omega})$ with $\lambda\in (0,\frac{1}{2}]$  by the Sobolev embedding \cite[Theorem 5.4]{Adams2003}.

Now, we consider the  first equation in \eqref{eq:P1}
\[
	-\Delta{u} +q(x)\phi u+|u|^{p-2}u=\omega u \text{ in } \Omega.
\]
We have that $u \in H^{1}_{0}(\Omega)$ is the unique solution of $\Delta{u} =(q(x)\phi-\omega) u+|u|^{p-2}u\in L^{2}(\Omega)$ because  $\phi \in C^{2,\lambda}(\overline{\Omega})$. Then, by  \cite[Theorem 9.9]{Gilbarg1988} it holds
\[
	\Delta{u} =(q(x)\phi-\omega) u+|u|^{p-2}u\in H^{2}_{0}(\overline{\Omega}).
\]
Therefore \cite[Theorem 8.10]{Gilbarg1988} implies that $\phi\in H^{4}_{0}(\overline{\Omega})$ which leads us to the fact that  $u \in C^{2,\lambda}(\overline{\Omega})$ with $\lambda\in \left(0,{1}/{2}\right]$ by  \cite[Theorem 5.4, Part II]{Adams2003}. Since $u\in H^{1}_{0}(\Omega)$ and $u \in C^{2,\lambda}(\overline{\Omega})$, $\lambda\in \left(0,{1}/{2}\right]$, we obtain 
\[
	-\Delta \psi =q(x)u^{2} \in H^{2}(\Omega).
\]
By  \cite[Theorem 8.10]{Gilbarg1988} it follows that
\[
	-a^{2}\Delta\phi+\phi=\psi \in H^{4}(\Omega),
\]
and then the interior regularity of $\phi$ increases by the same  Theorem, i.e. $\phi\in H^{6}(\Omega)$. Finally again  by  the Sobolev embedding \cite[Theorem 5.4, Part II]{Adams2003} 
\[
	\phi \in H^{6}(\Omega)\hookrightarrow C^{4,\lambda}(\overline{\Omega}),
\] 
where $\lambda\in \left(0,{1}/{2}\right]$, which shows the regularity of $\phi$.

\section{Dirichlet boundary conditions}
\label{sec:Dir}

We say that the triple $(u,\omega,\phi)\in B\times \mathbb R\times \mathbb{H}$ is a weak solution of \eqref{eq:P1} under conditions \eqref{eq:b1}-\eqref{eq:normalization_condition} and \eqref{eq:b5}-\eqref{eq:b4}, if
\begin{equation}
\int_{\Omega}\nabla u \nabla v dx+\int_{\Omega} q(x)\phi u v dx - \int_{\Omega}|u|^{p-2}uvdx= \omega \int_{\Omega} u vdx,\hspace{3mm}
\forall  v \in H^{1}_{0}\left(\Omega\right)
\label{variacion1}
\end{equation}
and
\begin{equation}
\int_{\Omega}\Delta\phi\Delta \xi dx+\int_{\Omega}\nabla \phi \nabla \xi  dx= \int_{\Omega}q(x) u^{2}\xi dx,\hspace{3mm} \forall  \xi\in
\mathbb{H}.\label{variacion2}
\end{equation}

\subsection{The variational setting}

Let us set also in this case the variational framework. Define the (natural) functional on $H^{1}_{0}(\Omega)\times\mathbb H$ given by
\begin{align}
    F \left(u,\phi\right)
    ={ }&{ }\frac{1}{2}\int_{\Omega} \left|\nabla u\right|^{2}dx + \frac{1}{2}\int_{\Omega}q(x) \phi u^{2}dx \notag \\
    &{ }-\frac{1}{p}\int_{\Omega} |u|^{p}dx-\frac{1}{4}\int_{\Omega} \left|\Delta \phi \right|^{2}dx-\frac{1}{4}\int_{\Omega} \left|\nabla \phi\right|^{2}dx. \label{funcionalgeneralizado}
\end{align}\label{prop.2}
Straightforward computations show that  $F$ is $C^{1}$ with derivatives given by
\begin{equation}
\partial_{u}F(u,\phi)\left[v\right]=\int_{\Omega}\nabla u \nabla vdx+ \int_{\Omega}q(x)uv\phi dx-\int_{\Omega}|u|^{p-2}u vdx , \quad
\forall v\in H^{1}_{0}(\Omega)\label{eqvar2}
\end{equation}
\begin{equation}
\partial_{\phi} F(u,\phi)\left[\xi\right]=\frac12\int_{\Omega} q (x)u^{2}\xi dx
-\frac{1}{2}\int_{\Omega}\Delta \phi\Delta \xi dx
-\frac12\int_{\Omega}\nabla \phi\nabla \xi dx, \quad
\forall \xi\in \mathbb{H}.\label{eqvar1}
\end{equation}

We have a first variational principle on the differentiable manifold $B$ where the functional is restricted.

\begin{theorem}\label{th:firstVP}
    The triple $\left( u,\omega,\phi\right)\in  H^{1}_{0}(\Omega)\times \mathbb{R} \times \mathbb H$ is a weak solution of \eqref{eq:P1} under conditions \eqref{eq:b1}-\eqref{eq:normalization_condition} and \eqref{eq:b5}-\eqref{eq:b4}, if, and only if, $\left(u,\phi\right)$  is a critical point of $F$ restricted to $B\times \mathbb H$ having $\omega$ as a Lagrange multiplier.
\end{theorem}

\begin{proof}
    The pair $\left(u, \phi\right)\in H^{1}_{0}(\Omega)\times \mathbb H$ is a critical point of  $F$ constrained to $B\times \mathbb H$ if and only if there exists a Lagrange multiplier $\omega \in \mathbb{R}$ such that
    $$\partial_{u}F(u,\phi) =\omega u\quad\text{and}\quad
    \partial_{\phi}F(u_,\phi)=0
    $$
    Taking into account  the expressions of the partial derivatives in \eqref{eqvar2} and \eqref{eqvar1} they are equivalent to \eqref{variacion1} and
    \eqref{variacion2},
    namely to say that
    $\left( \omega,u,\phi\right)\in \mathbb{R}\times H^{1}_{0}(\Omega)\times \mathbb H$ is a weak solution we were looking for.
\end{proof}

The functional $F$ in \eqref{funcionalgeneralizado} is unbounded both from above and below. Then the usual methods of Critical Point Theory cannot be directly applied. In order to deal with this issue, we shall reduce the functional in \eqref{funcionalgeneralizado} to the study of another functional depending on the single variable $u$, following a procedure introduced  by V. Benci and D. Fortunato in \cite{BenciFortunato1998} for these type of problems.

\begin{proposition}\label{prop:phia}
    For any  $u\in B$ the problem
     \begin{equation}
     \begin{cases}
    -\Delta \phi+ \Delta^{2}\phi  =q(x)  u^{2} & \mbox{ in } \Omega \\
     \Delta\phi=\phi=0  & \mbox{ on } \partial\Omega
     \label{eqsistem2}
     \end{cases}
    \end{equation}
    has a unique (and nontrivial) weak solution  $\Phi(u)\in \mathbb H$.
    Moreover it minimizes the functional
    \[
    	E(\phi)=\frac{1}{2}\int_{\Omega}|\nabla \phi|^{2}dx +\frac{1}{2}\int_{\Omega}|\Delta\phi|^{2}dx - \int_{\Omega} q(x)u^{2}\phi dx
    \]
    and $E(\Phi(u))<0$.
\end{proposition}

\begin{proof}
    Given any $u\in B$, we define the linear functional
    \begin{equation*}
    L_{u}:v\in \mathbb{H}\longmapsto \int_{\Omega}q(x)u^{2}v dx\in\mathbb{R}
    \end{equation*}
    which is also continuous since by the H\"older inequality and the Sobolev embedding
     we have for $v\in \mathbb H$, and suitable constants $c,c'>0$
    \begin{equation}
    \left|\int_{\Omega}q(x)u^{2}v dx\right|\leq|q|_{\infty} |u|^{2}_{4}|v|_{2}\leq c' |u|^{2}_{4}|\nabla v|_{2}\leq c\left\|u\right\|^{2}\left\| v \right\|.
    \label{Rieszcontinuo}
    \end{equation}
    Then, the functional $L_{u}$ is continuous, and by Riesz Representation Theorem, there exists a unique element,
    that we denote with   $\Phi(u) \in \mathbb{H}$
     such that
    $$L_{u}\left[v\right]=\left(\Phi(u),v \right) =
    \int_{\Omega} \nabla\Phi(u)\nabla vdx + \int_{\Omega}\Delta\Phi(u)\Delta v dx,
     \quad \forall v \in \mathbb{H}.$$
    In other words,  $\Phi(u) \in \mathbb{H}$ is the unique weak solution of \eqref{eqsistem2} and satisfies
    \begin{equation}\int_{\Omega} q(x)u^{2} vdx = \int_{\Omega}\Delta \Phi(u) \Delta v dx+ \int_{\Omega}\nabla\Phi(u)  \nabla v dx, \quad  \forall v \in\mathbb{H} . \label{aform}
    \end{equation}
    Finally it is standard to see that $\Phi(u)$ is the unique minimizer of $E$ and is nontrivial since the functional $E$ achieves strictly negative values near the origin and $E(0) = 0$.
\end{proof}

By taking $v=\Phi(u)$ in \eqref{aform} we obtain
\begin{equation} \label{aform2}
    \int_{\Omega} q(x)u^{2} \Phi(u)dx =\int_{\Omega}|\Delta \Phi(u)|^{2} dx+ \int_{\Omega}
    |\nabla\Phi(u) |^{2} dx=\|\Phi(u)\|^{2}.
\end{equation}
Due to  \eqref{Rieszcontinuo}, we have
\begin{equation}\label{eq:stima0}
     \int_{\Omega} q(x)u^{2} \Phi(u)dx
     \leq  c\|u\|^{2}\|\Phi (u)\|.
\end{equation}
As a consequence of  \eqref{aform2}  and \eqref{eq:stima0}we have the estimate
\begin{equation}\label{eq:stima}
    \|\Phi(u)\| \leq c\|u\|^{2}.
\end{equation}
Set now
\[
	\Gamma:=\left\{\left(u,\phi\right)\in H^{1}_{0}\left(\Omega\right)\times \mathbb{H}: \partial_{\phi} F\left(u,\phi\right)=0\right\}.
\]
Take the level set $B=\left\{u\in H^{1}_{0}(\Omega): |u|_{2}=1\right\}$ and define the map
\[
	\Phi: u\in B \longmapsto \Phi(u) \in \mathbb{H}
\]
where $\Phi(u)$ is the unique solution given in Proposition \ref{prop:phia}. Actually
\[
	\Phi(u)=(-\Delta+\Delta^{2})^{-1}(q(x)u^{2}),
\]
where $(-\Delta+\Delta^{2})^{-1}:\mathbb H'\to \mathbb H$  is the Riesz isomorphism.

\begin{proposition} \label{prop3}
    The map $\Phi$ is $C^{1}$ and $\Gamma$ is its  graph.
\end{proposition}

\begin{proof}
    By the Sobolev embedding, $H^{1}_{0}\left(\Omega \right)\hookrightarrow L^{6}\left(\Omega \right)$
     is continuous and it is easy to see that the map $u\mapsto u^{2}$ is $C^{1}$ from  $H^{1}_{0}\left(\Omega\right)$ into $L^{3}\left(\Omega\right)$ which is continuously embedded into $\mathbb{H}'$.
    Since the operator $(-\Delta+\Delta^{2})^{-1}$ is the Riesz isomorphism, it is also $C^{1}$. Therefore,
      the map $\Phi$ is also $C^{1}$, since is a composition of $C^{1}$ maps.
    
    Finally, the graph of $\Phi$ is
    \[
    	\textrm{Gr}(\Phi):=\left\{(u,\phi)\in M:\hspace{3mm} (-\Delta+\Delta^{2})^{-1}(q(x)u^{2})=\phi \right\}.
    \]
    Note that $\left(u,\phi\right)\in \textrm{Gr}(\Phi)$ means that $(-\Delta+\Delta^{2})\phi=q(x)u^{2}$ which is equivalent to say that
    $\partial_{\phi}F\left(u,\phi
    \right)=0$, which in turn is also equivalent to have  $\left(u,\phi \right)\in \Gamma$.
\end{proof}

Then  we can define the \textit{reduced functional }$J$ as follows
\begin{eqnarray*}
    J (u)&:=&F \left(u,\Phi(u)\right).
\end{eqnarray*}
From \eqref{aform2} we deduce
\begin{equation*}
    \frac{1}{4}\int_{\Omega}\left|\Delta \Phi (u)\right|^{2}dx+\frac{1}{2}\int_{\Omega} \left|\nabla \Phi (u)\right|^{2}dx=\frac{1}{2}\int_{\Omega}q(x)u^{2}\Phi (u)
    dx -\frac{1}{4}\int_{\Omega} \left|\Delta \Phi (u)\right|^{2}dx
\end{equation*}
and hence the  functional $J $ can be written as
\begin{align}
    J (u)
    &=\frac{1}{2}\int_{\Omega}\left|\nabla u\right|^{2}dx
        - \frac{1}{p} |u|^{p}_{p}
        + \frac{1}{4}\int_{\Omega}\left|\Delta \Phi (u)\right|^{2}dx
        + \frac{1}{2}\int_{\Omega}\left|\nabla \Phi (u)\right|^{2}dx
        - \frac{1}{4}\int_{\Omega}\left|\nabla \Phi (u)\right|^{2}dx \nonumber \\
    &= \frac{1}{2}\int_{\Omega}\left|\nabla u\right|^{2}dx
        -\frac{1}{p} |u|^{p}_{p}
        + \frac{1}{4}\int_{\Omega}\left|\Delta \Phi (u)\right|^{2}dx
        + \frac{1}{4}\int_{\Omega}\left|\nabla \Phi (u)\right|^{2}dx
        \label{eq7u}.
\end{align}

The functional $J $ is then
 bounded from below, even and, by Proposition \ref{prop3}, $C^{1}$.
Then, the Fr\'echet derivative of  $J $ at $u$ is given by
\begin{equation}
J '(u)=\partial_{u}F (u,\Phi (u))+\partial_{\phi}F (u,\Phi (u))\Phi '(u)
=\partial_{u}F (u,\Phi (u))
\label{derivtive3.30}
\end{equation}
as  a linear and continuous operator on $H^{1}_{0}(\Omega)$.
Taking into account \eqref{eqvar2}, we get
\begin{equation}\label{eq:J'}
    J '(u)[v] = \int_{\Omega}\nabla u \nabla vdx+\int_{\Omega} q(x)u v\Phi (u)dx
    -\int_{\Omega} |u|^{p-2}uvdx,
    \quad \forall v\in H^{1}_{0}(\Omega).
\end{equation}

From Theorem \ref{th:firstVP},  we are reduced to finding critical points $(u ,\phi )$ of $F $ on $B\times \mathbb H$ with the associated Lagrange multiplier $\omega $. The following is a second variational principle and describes the relation between critical points of  $F $ on $B\times \mathbb H$ and critical points of $J $ restricted to $B$.

\begin{proposition}\label{prop4}
     Let $\left(u , \phi \right)\in B\times \mathbb H $ and $ \omega  \in \mathbb{R}$.  The following statements are equivalent.
    \begin{enumerate}
    \item \label{a-prop4} The pair $\left(u ,\phi \right)$ is a critical point of $F $ constrained to
    $B\times \mathbb H$ having $\omega $ as Lagrange multiplier.
    \item  \label{b-prop4} The function $u $ is a critical point of $J $ constrained to $B$ having $\omega $ as Lagrange multiplier and
    $\phi =\Phi (u ).$
    \end{enumerate}
\end{proposition}

\begin{proof}
    Condition \eqref{a-prop4} means that
    $$
        \partial_{u}F (u ,\phi )=\omega u \quad \text{and}\quad \partial_{\phi}F (u ,\phi )=0.
    $$
    It follows from Proposition \ref{prop3}  that $\phi =\Phi (u )$ and by \eqref{derivtive3.30}, $J' (u )=\omega u .$ This is exactly \eqref{b-prop4}. On the other hand, \eqref{b-prop4} implies
    \[
    	J' (u ) = \omega  u  \quad \text{and}\quad (u ,\Phi (u ))\in \textrm{Gr}(\Phi )
    \]
    and then $\partial_{\phi}F (u ,\phi )=0$. Consequently, again by \eqref{derivtive3.30}, we infer
    \[
    	\omega u  = J' (u ) = \partial_{u}F (u ,\Phi (u ))
    \]
    so \eqref{a-prop4} is proved.
\end{proof}

In particular the above result says that all the solutions are of type $( u, \omega , \Phi (u))$.

In view of the previous result, for brevity we may refer just to the unknown  $u$ as a solution of the system ($\omega$ and $\phi$ are then univocally determined), and $J $ to its energy.

The compactness condition is satisfied as we show now.

\begin{lemma} \label{condPS}
    The functional $J $  constrained to $B$ satisfies the Palais-Smale condition.
\end{lemma}

\begin{proof}
    Let $\{w_{n}\}\subset B$ be a Palais-Smale sequence for $J $. Then, there exist two sequences $\{\lambda_{n}\}\subset \mathbb{R}$ and
    $\{\varepsilon_{n}\}\subset H^{-1}(\Omega)$, where $H^{-1}(\Omega)$ is the dual space of $ H^{1}_{0}(\Omega)$, such that $\varepsilon_{n}\to
    0$
    and, see \eqref{eq7u},
    \begin{eqnarray}\label{eq.der.J.ps}
    J '(w_{n})=\lambda_{n}w_{n}+\varepsilon_{n}, \\
    J (w_{n}) =  \frac{1}{2}\|w_{n}\|^{2} + \frac{1}{4}\|\Phi (w_{n})\|^{2}
    -\frac{1}{p}\int_{\Omega} |w_{n}|^{p}dx\longrightarrow c. \nonumber\label{eq.der.J.ps2}
    \end{eqnarray}
    In particular $\{w_{n}\}$ and $\{\Phi (w_{n})\}$ are bounded respectively in $H^{1}_{0}(\Omega)$ (being $p<10/3$) and~$\mathbb H$.
    
    By \eqref{eq:J'} and \eqref{eq.der.J.ps}, after the usual integration by parts, we get
    \begin{equation*}
    \int_{\Omega}\left|\nabla w_{n}\right|^{2}dx+\int_{\Omega}q(x)w^{2}_{n}\Phi (w_{n})dx
    -\int_{\Omega}|w_{n}|^{p}dx
    = \lambda_{n}+ \langle\varepsilon_{n},w_{n}\rangle.
    \end{equation*}
    Using the boundedness of $\{w_{n}\}$, \eqref{eq:stima0} and the fact that $\varepsilon_{n}\to 0$, we see that also $\{\lambda_{n}\}$ has to be bounded.
    
     The equation \eqref{eq.der.J.ps} is rewritten as
    $-\Delta w_{n}+q(x)w_{n}\Phi (w_{n})-|w_{n}|^{p-2}w_{n}-\lambda _{n}w_{n}=\varepsilon_{n}$
    and applying the inverse Riesz isomorphism $\Delta^{-1}:H^{-1}(\Omega)\rightarrow H^{1}_{0}(\Omega)$, we obtain that
    \begin{equation}
     w_{n}=\Delta^{-1}\left(q(x)w_{n}\Phi (w_{n})\right)
     -\Delta^{-1}(|w_{n}|^{p-2}w_{n})-\lambda _{n}\Delta^{-1}w_{n}-\Delta^{-1}\varepsilon_{n},
    \label{principal}
    \end{equation}
    and $\{\Delta^{-1}\varepsilon_{n}\}$ is a convergent sequence.
    Now $\{q(x)w_{n}\Phi (w_{n})\}$ is bounded in $L^{2}(\Omega)$ due to  the estimates
    $$\int_{\Omega}|qw_{n}\Phi (w_{n})|^{2} dx \leq |q|^{2}_{\infty}|\Phi (w_{n})|_{4}^{2}|w_{n}|_{4}^{2}\leq c' \|\Phi (w_{n})\|^{2} |w_{n}|_{4}^{2},$$
    as well as $\{|w_{n}|^{p-2}w_{n}\}$, being $2(p-1)\in(2,14/3)$.
    Then these sequences are also bounded in $H^{-1}(\Omega)$. Actually since $\Delta^{-1}$ is compact, we deduce that (up to subsequences)
     $$\Big\{\Delta^{-1}\left(q(x) w_{n}\Phi (w_{n})\right)\Big\}, \  \
     \Big\{\Delta^{-1}\left(|w_{n}|^{p-2}w_{n}\right)\Big\}, \  \
     \Big\{\lambda _{n}\Delta^{-1}w_{n}\Big\}
     \quad \text{ are convergent.
     }
     $$
    Coming back to \eqref{principal}, we infer that $\{w_{n}\}$ is convergent (up to subsequences) in $H^{1}_{0}(\Omega)$, and the limit is of course in $B$.
\end{proof}

\subsection{Proof of the main result} \label{subsection3.6}

We show  that the functional $J $ restricted to $B$ has infinitely many critical points. Let $n$ be a positive integer. By Lemma~\ref{lem:sublevels_have_finite_genus}, there exists a positive integer $k=k(n)$ such that 
\[
	\gamma\left(J ^{n}\right)=k.
\]
Now, consider the collection
\begin{equation*}
    \mathcal{A}_{k+1}:=\left\{A\subset B: A \text{ is symmetric and closed with }\gamma\left(A\right)\geq k+1 \right\}.
\end{equation*}
By Lemma \ref{lemma7}, there exists a compact set $\textsc{K}\subset B$ such that $\textsc{K}\in\mathcal{A}_{k+1}$, then $\mathcal{A}_{k+1}\neq\emptyset.$

Since by the definitions
\[
	\gamma\left(A\right)>\gamma\left( J ^{n}\right), \text{ for all } A\in\mathcal{A}_{k+1},
\]
by the monotonicity property of genus $A\not\subset J ^{n}$, thus  
\[
	\sup J \left(A\right) >n, \text{ for all } A\in\mathcal{A}_{k+1}.
\]
Consequently
\[
	b_{n}:=\inf\left\{\sup J (A): A\in \mathcal{A}_{k+1} \right\}\geq n.
\]
We know by Lemma \ref{condPS} that $ J $ satisfies the Palais-Smale condition on $B$ and it is an even functional.
Then it follows from \cite[Theorem 5.7]{Struwe} that
\[
	b_{n} \text{ is a critical value of  $J $  on $B$}
\]
achieved on some $u_{n}\in B$. By the Lagrange Multipliers Theorem, for any $n\in \mathbb N$ there exist $\omega_{n}\in \mathbb{R}$ such that
\begin{equation*}
J '\left(u_{n}\right)=\omega_{n}u_{n} \quad\text{ with }\quad J \left(u_{n}\right)=b_{n}\geq n.\label{MLag}
\end{equation*}
Now evaluating $J '\left(u_{n}\right)=\omega_{n}u_{n}$ on the same $u_{n}$ we find
\begin{equation}
    \frac{1}{2}\int_{\Omega}\left|\nabla u_{n}\right|^{2}dx+\frac{1}{2}\int_{\Omega}q(x)\Phi (u_{n})u_{n}^{2}dx=\frac{1}{2}\omega_{n}.\label{eq53}
\end{equation}
In particular $\omega_{n}>0$. Replacing the above equation in the functional, which is given by
\begin{equation}\label{eq:fauan}
    J (u_{n}) = \frac{1}{2}\int_{\Omega}|\nabla u_{n}|^{2}dx
    +\frac{1}{4}\int_{\Omega} |\Delta \Phi (u_{n})|^{2}dx
    +\frac{1}{4}\int_{\Omega} |\nabla \Phi (u_{n})|^{2}dx,
\end{equation}
we have
\begin{equation*}
    b_{n}=J \left(u_{n}\right)=\frac{1}{2}\omega_{n}
    -\frac{1}{4}\int_{\Omega}\left|\Delta \Phi (u_{n})\right|^{2}dx-\frac{1}{4}\int_{\Omega}\left|\nabla \Phi (u_{n})\right|^{2}dx
    \label{eq54}
\end{equation*}
or
\begin{equation}\label{eq:omega}
    \omega_{n}=2b_{n}+\frac{1}{2}\int_{\Omega}\left|\Delta \Phi (u_{n})\right|^{2}dx+\frac{1}{2}\int_{\Omega}\left|\nabla
    \Phi (u_{n})\right|^{2}dx>2n
\end{equation}
which shows that $\omega_{n}\rightarrow \infty$ as $n\rightarrow \infty$. We note that \eqref{eq:omega} implies also that
\begin{equation*}
    \omega_{n}\geq \frac{1}{2}\|\Phi (u_{n})\|^{2} .
\end{equation*}

Recalling \eqref{eq:stima}, we rewrite \eqref{eq53} as
\begin{eqnarray*}
    \omega_{n}&=&\int_{\Omega}\left|\nabla u_{n}\right|^{2}dx
    +\int_{\Omega}\left|\Delta \Phi (u_{n})\right|^{2}dx + \int_{\Omega}\left|\nabla \Phi (u_{n})\right|^{2}dx\\
    &=& \left\|u_{n}\right\|^{2} +\|\Phi (u_{n})\| ^{2}\\
    &\leq& \left\|u_{n}\right\|^{2} + C\|u_{n}\|^{4}
\end{eqnarray*}
and then $\left\|u_{n}\right\|\to +\infty$ as $n\to+\infty.$

Summing up, we have found for any $n\in \mathbb N$:
$$
    u_{n}\in B\subset H^{1}_{0}(\Omega), \quad \phi_{n}:=\Phi (u_{n})\in \mathbb H, \quad \omega_{n}\in \mathbb R
$$
solutions as stated in Theorem \ref{tma1.1}. Furthermore the above computations provide the additional information and estimates on the norm of the solutions and the energy levels of the functional:
\begin{enumerate}
    \item $J (u_{n})=\frac12\|u_{n}\|^{2}+\frac14\|\phi_{n}\| ^{2}\geq n$, \smallskip
    \item $\omega_{n}=\|u_{n}\|^{2}+\|\phi_{n}\| ^{2}>2n$, \smallskip
    \item $\|\phi_{n}\|^{2}  \leq 2\omega_{n}$, \smallskip
    \item $\|\phi_{n}\| \leq C \|u_{n}\|^{2}.$
\end{enumerate}
It is well known that $u_{1}$ is the minimum of  $J $, for this reason we say that $(\omega_{1},u_{1}, \phi_{1})$ is a  ground state solution. Correspondingly, $b_{1}$ is the ground state level. Observe that since $J (|u|) = J (u)$, the ground state $u_{1}$ can be assumed positive.
The proof of Theorem \ref{tma1.1} is thereby complete.

\section{Neumann boundary conditions} \label{sec:aux}

\subsection{An auxiliary problem}

In order to deal with homogeneous boundary conditions, that will permit to write the functional in a simpler form, we make a change of variables.

Consider the following auxiliary problem (where $\alpha$ is defined in \eqref{eq:alfa})
$$
    \begin{cases}
     - \Delta \chi+\Delta^2 \chi = \alpha/|\Omega| & \text{ in } \ \Omega,\\
    \partial_{\mathbf n}{\chi} = h_1, \ \ \partial_{\mathbf n}\Delta\chi = h_2 \quad & \text{ on } \ \partial \Omega,  \\
    \displaystyle\int_{\Omega} \chi  dx = 0. \label{eq:chi4}
    \end{cases}
$$
It is easy to see it has a  solution, see e.g., \cite{Taylor}. The unicity is guaranteed due to the null-average condition.

Indeed, let $\chi_1, \chi_2$ be two solutions of the above system. Then $w = \chi_1 - \chi_2$ satisfies
$$
    -\Delta w +\Delta^{2}w= 0 \ \text{ in } \Omega, \qquad \partial_{\mathbf n}w=
    \partial_{\mathbf n}\Delta w = 0 \ \text{ on } \partial \Omega.
$$
By multiplying by $w$ and integrating by parts we have
$$
    \int_{\Omega} |\nabla w|^2   dx + \int_{ \Omega} |\Delta w|^{2}dx = 0
$$
hence $\nabla w = 0$ and thus $w$ is a constant. Hence, there is uniqueness up to a constant but only one of them has null integral; this implies $\chi_{1} = \chi_{2}$.

The change of variables we make is
$$
    \varphi = \phi - \chi - \mu, \quad \text{where }\ \ \mu = \frac{1}{|\Omega|}\int_{\Omega} \phi \, dx.
$$
In the new variables $(u, \omega, \varphi, \mu)$ our problem \eqref{eq:P1} with conditions \eqref{eq:b1}-\eqref{eq:normalization_condition} and \eqref{eq:b2}-\eqref{eq:b3}, can be rewritten as
\begin{equation}\label{eq:P}
    \begin{cases}
    - \Delta u + q(x)(\chi + \varphi)u -  |u|^{p - 2}u  = \omega u - \mu q(x) u & \text{ in } \ \Omega,\\
     - \Delta \varphi+\Delta^2 \varphi  =  q(x)u^2 - \alpha/|\Omega| & \text{ in }  \Omega,   \\
    u  =  0 & \text{ on } \partial \Omega, \\
    \displaystyle\int_{\Omega} u^2  dx  =  1,   \\
    \partial_{\mathbf n}\varphi  =  \partial_{\mathbf n}\Delta \varphi  = 0 & \text{ on } \partial \Omega,   \\
    \displaystyle\int_{\Omega} \varphi  dx  =  0
    \end{cases}
\end{equation}
and the compatibility condition between the second equation in the system and the boundary conditions on $\varphi$ reads as
$$
    \int_{\Omega} q(x)u^2 dx = \alpha.
$$

In order to formulate the variational problem it is necessary to introduce the set where the functional will be defined. As we will see, in this case it is not simply the $L^{2}$ sphere $B$.

\subsection{The manifold $M$} \label{sec:M}

We introduce the following sets
\begin{align*}
    B & := \left\{ u \in H_0^1(\Omega) \ : \ \int_{\Omega} u^2  dx = 1 \right\}, \\
    N & := \left\{u \in H_0^1(\Omega) \ : \ \int_{\Omega} q(x)u^2  dx = \alpha \right\}, \\
    M & := B \cap N
\end{align*}
and observe that if the problem has a solution, then, by definition, $M \neq \emptyset$. Hence we easily have
\begin{equation}
    \label{eq:range_for_alpha}
    q_{\min} \leq \alpha \leq q_{\max}
\end{equation}
where
$$
    q_{\min} = \inf_\Omega q \quad \text{ and } \quad q_{\max} = \sup_\Omega q.
$$
Indeed, if  $\alpha < q_{\min}$, then
$$
    \alpha = \int_{\Omega} q(x)u^2  dx \geq  q_{\min} > \alpha,
$$
which is a contradiction. Analogously for the case $\alpha > q_{\max}$. Moreover  from \eqref{eq:range_for_alpha} we deduce that $q^{-1}(\alpha)$ is not empty, and indeed what is important is  its Lebesgue measure. Let us see this in some details.

Suppose $\alpha = q_{\min}$ and $|q^{-1}(\alpha)| = 0$. Then
\begin{align*}
    \int_{\Omega} q(x)u^2  dx &= \int_{\{ x\in \Omega : q(x)>\alpha\}} q (x)u^2  dx >\alpha,
\end{align*}
so $M=\emptyset$. If $\alpha = q_{\max}$ and $|q^{-1}(\alpha)| = 0$ we proceed in an analogous manner to conclude that $M$ is empty and so the problem has no solutions. Therefore, we arrive at the following necessary condition for the existence of solutions: either
\begin{equation}
    \label{eq:necessary_condition_1}
    q_{\min} < \alpha < q_{\max}
\end{equation}
or
\begin{equation}
    \label{eq:necessary_condition_2}
    |q^{-1}(\alpha)| \neq 0.
\end{equation}

We now state some properties of the set $M$, referring the reader to \cite{PisaniSiciliano2013} for the omitted proofs.

First note that $M$ is symmetric with respect to the origin: if $u \in M$, then $-u \in M$. This follows trivially from the definition of $M$. We also note that $M$ is weakly closed in $H_0^1 (\Omega)$; this is due to the compact embedding of $H^{1}_{0}(\Omega)$ into $L^2(\Omega)$ being $\Omega $ bounded.

However we have the problem to show that this set is not empty. We want to show that  condition \eqref{eq:necessary_condition_1} will guarantee this fact.  For this, we introduce the following notation. Let $A \subset \Omega$ be an open subset and define
$$
    B_A := = \left\{u \in H_0^1 (A) \ : \ \int_A u^2  dx = 1\right\}
$$
and
$$
    g_A: u \in B_A \mapsto \int_A q(x)u^2  dx \in \mathbb R.
$$

It is immediately seen that
\begin{equation*}
    \label{eq:inclusion_1}
    g_A(B_A) \subset [\inf_A q, \sup_A q].
\end{equation*}

\begin{lemma}
    The following inclusion holds:
    \begin{equation*}
        \label{eq:inclusion_2}
        (\inf_A q, \sup_A q) \subset \overline{g_A(B_A)}.
    \end{equation*}
\end{lemma}

\begin{proof}
    See \cite[Lemma 2.1]{PisaniSiciliano2013}.
\end{proof}

Then  we can conclude the following:

\begin{proposition} \label{prop:M_is_not_empty}
    Let $A \subset \Omega$ be an open subset. If $\alpha \in (\inf_A q, \sup_A q)$ then there exists $u \in H_0^1(A)$ such that
    $$
    \int_A u^2  dx = 1 \quad \text{ and } \quad \int_A q(x)u^2  dx = \alpha.
    $$
\end{proposition}

In particular by taking $A=\Omega$ we get
\begin{proposition}
    Assume that $\inf_\Omega q < \alpha < \sup_\Omega q$. Then $M$ is not empty.
\end{proposition}

\begin{proof}
    See \cite[Proposition 2.2]{PisaniSiciliano2013}.
\end{proof}

We state an important topological property of $M$:
it has subsets of arbitrarily large genus.

\begin{proposition} \label{thm:existence_of_sets_of_genus_k}
    Let $u_1, \ldots, u_k \in M$ be functions with disjoint supports and let
    $$
    V_k = \langle u_1, \ldots, u_k \rangle
    $$
    be the space spanned by $u_1, \ldots, u_k$. Then $M\cap V_k$ is the $(k-1)$-dimensional sphere, hence $\gamma(M \cap V_k) = k$.
\end{proposition}

The proof is given in \cite[Theorem 2.1]{PisaniSiciliano2007}. However we write it here for the reader convenience.

\begin{proof}
    Since the functions $u_1,\dots,u_k$ are linearly independent, the space $V_{k}$ has dimension $k$.
    Since  $B$ is the $L^{2}-$sphere in $H_{0}^{1}(\Omega)$,
    it is immediately seen that $B\cap V_{k}$ is just the unit sphere in $V_k$ with respect to the $L^2-$norm
    (see Lemma \ref{lemma7}).
    Let us prove that
    \[
    	M\cap V_{k}=B\cap V_{k}.
    \]
    
    It is obvious that $M\cap V_{k}\subset B\cap V_{k}$. On the other hand,
    if $u=\sum_{i=1}^{k}d_{i} u_{i}\in B\cap V_{k}$, then
    $$
    1=|u |_2^2 = \sum_{i=1}^{k}d_{i}^{2}.
    $$
    Therefore
    $$
    \int_{\Omega} q(x)u^{2}dx=\int_{\Omega} q(x)\sum_{i=1}^{k}d_{i}^{2}u_{i}^{2}dx=\sum_{i=1}^{k}d_{i}^{2}\int_{\Omega} q(x)u_{i}^{2}dx=
    \alpha \sum_{i=1}^{k}d_{i}^{2} =\alpha,
    $$
    that is $u\in M$.
\end{proof}

Now, it is natural if one raises the question if whether there exists such functions with disjoint supports for arbitrary $k$. The answer is affirmative under the assumption on $\alpha$ given in Theorem~\ref{thm:main_theorem}.

\begin{proposition}
    If \eqref{eq:necessary_condition_1} holds then for every $k \geq 2$ there exist $k$ functions $u_1, \ldots, u_k \in M$ with disjoint supports. Hence $\gamma(M) = + \infty$.
\end{proposition}

\begin{proof}
    By \eqref{eq:necessary_condition_1}, the subsets
    \[
    	\Omega_{+}=\{x \in \Omega: q(x)>\alpha\}\,, \ \ \ \Omega_{-}=\{x\in \Omega: q(x)<\alpha\}
    \]
    are open and not empty. We can choose $2k$ disjoint balls, $\{Y_1, Y_2, \dots,Y_k\}\subset \Omega_{-}$ and
    $\{Z_1, Z_2 \dots,Z_k\}\subset \Omega_{+}$,
    then we set
    \[
    	A_i = Y_i \cup Z_i\,, \ \ \ i=1,...,k.
    \]
    It follows by construction that
    $$
    \inf q(A_i) < \alpha < \sup q(A_i).
    $$
    Therefore by the previous Proposition \ref{thm:existence_of_sets_of_genus_k} we find $u_i\in H^1_0(A_i)$ such that
    \[
    	\int_{\Omega} u_i^{2}\, dx=\int_{A_i} u_i^{2}\, dx=1\quad \text{and} \quad \int_{\Omega}qu_i^{2}\, dx=\int_{A_i} qu_i^{2}\, dx=\alpha
    \]
    (of course we identify every function in $H^1_0(A_i)$ with its trivial extension).
    All these functions $u_{1},u_{2},...,u_{k}\in M$ have disjoint supports.
\end{proof}

In order to use  classical variational methods, we explore the differential structure of $M$. Let
\begin{align*}
    & G_1 : u \in H_0^1 (\Omega) \mapsto \int_{\Omega} u^2  dx - 1 \in \mathbb R, \\
    & G_2 : u \in H_0^1(\Omega) \mapsto \int_{\Omega} q(x)u^2  dx - \alpha \in \mathbb R
\end{align*}
and
$$
G := (G_1, G_2)
$$
which is of class $C^{1}$. Then
$$
    M = \left\{u \in H_0^1(\Omega) \ : \ G_1(u) = G_2(u) = 0 \right\} = G^{-1}(0).
$$

Let us show, for the reader convenience, that $G_1'(u)$ and $G_2'(u)$ are linearly independent, so $G$ will be a submersion and $M$ a submanifold of codimension $2$.

\begin{proposition}
    Assume that $M$ is not empty.
    If
    \begin{equation}
    \label{eq:measure_is_zero}
    |q^{-1}(\alpha)| = 0,
    \end{equation}
    then the differentials $G_1'(u)$ and $G_2'(u)$ are linearly independent for every $u \in M$.
\end{proposition}

\begin{proof}
    Suppose that there are $a, b \in \mathbb R$ such that
    $$
    a G_1'(u) + bG_2'(u) = 0 \quad \text{ in } H^{-1} (\Omega)
    $$
    for some $u \in M$. Evaluating this expression in $u$, we find that $a + b \alpha = 0$. Then
    $$
    aG_1'(u)[v] + b G_2'(u)[v] = b \del{-\alpha \int_{\Omega} uv  dx + \int_{\Omega} q(x)uv  dx } = 0 \quad \forall v \in H_0^1(\Omega),
    $$
    that is,
    $$
    b \int_{\Omega} (q(x) - \alpha)uv  dx = 0 \quad \forall v \in H_0^1 (\Omega).
    $$
    If $b \neq 0$ then we would have $(q(x) - \alpha)u = 0$ a.e., and hence, in view of \eqref{eq:measure_is_zero}, $u = 0$, a contradiction. Thus
    $G_1'(u)$
    and $G_2'(u)$ are linearly independent for all $u \in M$.
\end{proof}

\subsection{The variational setting}

We now proceed to study the variational framework of the problem. Our aim is to construct a functional whose critical points on $M$ will be the weak solutions of the problem.

Following \cite{GazzolaGrunauSweers2010}, let
$$
    V = \left\{\xi \in H^2(\Omega) \ : \ \partial_{\mathbf n}\xi = 0 \text{ on } \partial \Omega \right\}
$$
which is  a closed subspace of $H^2(\Omega)$. Indeed, let $\{v_n\} \subset V$ such that $v_n \to v$ in $V$. Then $0 = \gamma_1(v_n) \to \gamma_1(v)$ and hence $\gamma_1(v) = 0$, where $\gamma_1$ denotes the trace operator which for smooth functions gives the directional derivative in the direction of the exterior normal on the boundary. Being a closed subspace, $V$ inherits the Hilbert space structure of $H^2(\Omega)$.

Recall that the change of variables was
$$
\varphi = \phi - \chi - \mu
$$
where
$$
\mu = \frac{1}{|\Omega|} \int_{\Omega} \phi  dx.
$$
In this way, we have $\overline \varphi = 0$, where from now on, we denote the average of $f$ with  $\overline f$. Consider then the following natural decomposition of $V$:
\begin{equation}
\label{eq:decomposition_of_V}
V = \widetilde V \oplus \mathbb R
\end{equation}
where
$$
\widetilde V = \left\{\eta \in V \ : \ \overline \eta = 0\right\}.
$$
On $\widetilde V$ (which is an Hilbert space being closed) we have the equivalent (squared) norm
$$
\|\eta\|_{\widetilde V}^{2} = \del{|\nabla \eta|_2^2 + |\Delta \eta|_2^2}.
$$

Consider the functional $F:H_0^1(\Omega) \times H^2 (\Omega)$ defined as follows:
\begin{align*}
F(u, \varphi) = & \frac12 \int_{\Omega} |\nabla u|^2  dx + \frac12 \int_{\Omega} q(x)(\varphi + \chi)u^2  dx - \frac{ 1}{p} \int_{\Omega} |u|^p  dx\\
& - \frac{1}{4} \int_{\Omega} (\Delta \varphi)^2  dx - \frac{1}{4} \int_{\Omega} |\nabla \varphi|^2  dx - \frac{\alpha}{2 |\Omega|} \int_{\Omega} \varphi dx
\end{align*}
which is easily seen to be of class $C^1$. For every  $u \in H_0^1(\Omega)$ and $\varphi \in H^2 (\Omega)$ we have the following partial derivatives:
\begin{align*}
F'_u(u, \varphi)[v] & = \int_{\Omega} \nabla u  \nabla v  dx + \int_{\Omega} q(x)(\varphi + \chi)uv  dx  -   \int_{\Omega} |u|^{p - 2}u v  dx, \\
F'_\varphi(u, \varphi)[\xi] &= \frac12 \int_{\Omega} q(x) \xi u^2  dx - \frac12 \int_{\Omega} \Delta \varphi \Delta \xi  dx - \frac12 \int_{\Omega} \nabla
\varphi
 \nabla \xi  dx - \frac{\alpha}{2|\Omega|}\int_{\Omega} \xi  dx
\end{align*}
for every $v \in H_0^1(\Omega)$ and $\xi \in H^2(\Omega)$.

Hence, $(u, \varphi, \omega, \mu)\in H^1_0(\Omega)\times H^2(\Omega)\times \mathbb R\times \mathbb R$ is a weak solution of \eqref{eq:P} if and only if
\begin{equation}\label{eq:defsol}
\begin{array}{ll}
& (u, \varphi) \in M \times \widetilde V, \\
& \forall v \in H_0^1 (\Omega):  \ F_u'(u, \varphi)[v] =\displaystyle \omega \int_{\Omega} uv  dx - \mu \int_{\Omega} q(x) u v  dx, \\
& \forall \xi \in V: \ F_\varphi'(u, \varphi)[\xi] = 0.
\end{array}
\end{equation}

\begin{theorem}
Let $(u, \varphi) \in H_0^1(\Omega) \times H^2 (\Omega)$. Then there exist $\omega, \mu \in \mathbb R$ such that $(u, \varphi, \omega, \mu)$ is a solution of \eqref{eq:P}
if and only if $(u, \varphi)$ is a critical point of $F$ constrained on $M \times
\widetilde V$, and in this case the real constants $\omega, \mu$ are the two Lagrange multipliers with respect to $F'_u$.
\end{theorem}

\begin{proof}
Note that the tangent space to $\widetilde V$ at $\varphi$ is $\widetilde V$ itself. Now, $(u, \varphi)$ is a critical point of $F$ constrained on $M \times \widetilde V$ if,
and only if
\begin{align*}
& \forall v \in T_uM:  \ F_u'(u, \varphi)[v] = 0, \\
& \forall \xi \in \widetilde V: \ F_\varphi'(u, \varphi)[\xi] = 0.
\end{align*}
Then $(u,\varphi)$ is a weak solution of \eqref{eq:P} according to \eqref{eq:defsol} and the Lagrange multipliers rule. So it is a constrained critical point for the functional $F$.

Suppose on the contrary that $(u, \varphi)$ is a constrained critical point. Then, again by the  Lagrange multipliers rule, we have that there exist $\omega, \mu\in \mathbb R$ such that
$$
\forall v\in H^1_0(\Omega): F_u'(u, \varphi)[v] = \omega \int_{\Omega} uv  dx - \mu \int_{\Omega} q(x)uv  dx.
$$
It remains to prove that $F_\varphi'(u, \varphi)[\xi] = 0$ for all $\xi \in V$. But this follows by the decomposition \eqref{eq:decomposition_of_V}, noticing that $F_\varphi'(u, \varphi)[r] = 0$ for every constant $r \in \mathbb R$. Then \eqref{eq:defsol} is satisfied and this concludes the proof.
\end{proof}

The functional $F$ constrained on $M \times \widetilde V$ is unbounded from above and from below. This issue has been addressed  for the first time in \cite{BenciFortunato1998} for a similar problem, and solved by following the steps below:

\begin{enumerate}
\item[(i)] for every fixed $u\in H^{1}_{0}(\Omega)$ there exists a unique $\Phi(u)$ such that $F'_\varphi(u,  \Phi(u)) = 0$;

\item[(ii)] the map $u \mapsto \Phi(u)$ is of class $C^1$;

\item[(iii)] the graph of $\Phi$ is a manifold, and we are reduced to study the functional $J(u) = F(u, \Phi(u))$, possibly constrained.
\end{enumerate}

However in our case we cannot argue exactly in the same way since the method sketched above fails. First of all, we see that $F_\varphi'(u, \varphi) = 0$ with $\varphi \in \widetilde V$ is equivalent to
$$
\begin{cases}
\Delta^2 \varphi - \Delta \varphi - q(x)u^2 + \alpha/|\Omega| = 0 & \text{ in } \ \Omega, \\
\partial_{\mathbf n} \varphi=
\partial_{\mathbf n}\Delta \varphi = 0 & \text{ on } \ \partial \Omega, \\
\displaystyle\int_{\Omega} \varphi  dx = 0
\end{cases}
$$
and such a problem does not always have a unique solution. In fact this happens if and only if $u \in N$, due to the compatibility condition. Secondly, since $N$ is not a manifold (unless $\alpha \neq 0$) we cannot require the map $\Phi: u \mapsto \Phi(u)$ to be of class $C^1$ in $N$. The idea is then to  extend such a map $\Phi$, possible in view of the next two results.

\begin{proposition}
\label{prop:auxiiary_problem_to_extend_Phi}
For every $w \in L^{6/5}(\Omega)$ there exists a unique $L(w) \in \widetilde V$ solution of
$$
\begin{cases}
- \Delta \varphi +\Delta^2 \varphi - q(x)w + \overline w = 0  & \text{ in } \ \Omega, \\
\partial_{\mathbf n}\varphi=
\partial_{\mathbf n}\Delta \varphi = 0 & \text{ on } \ \partial \Omega, \\
\displaystyle\int_{\Omega} \varphi  dx = 0
\end{cases}
$$
and the map $L: L^{6/5}(\Omega) \longrightarrow \widetilde V$ is linear and continuous, hence of class $C^{\infty}$.
\end{proposition}

\begin{proof}
The weak solutions to the problem are functions $\varphi$ in the Hilbert space  $\widetilde V$ such that
$$
\int_{\Omega} \Delta \varphi \Delta v  dx + \int_{\Omega} \nabla \varphi  \nabla v  dx - \int_{\Omega} q(x)wv  dx = 0 \qquad \forall v \in \widetilde V,
$$
so the result follows by applying the Riesz Theorem, in fact the bilinear form $b: \widetilde V \times \widetilde V \to \mathbb R$ given by
$$
b(\varphi, v) = \int_{\Omega} \Delta \varphi \Delta v  dx + \int_{\Omega} \nabla \varphi  \nabla v  dx.
$$
is nothing but the  scalar product in $\widetilde V$.
\end{proof}

From well-known properties of Nemytsky operators (see e.g., \cite{AmbrMalc}) the following proposition holds.

\begin{proposition}
The map
$$
u \in L^6(\Omega) \mapsto q(x)u^2 \in L^{6/5}(\Omega)
$$
is of class $C^1$.
\end{proposition}

As a consequence the  map
$$
\Phi: u \in H_0^1 (\Omega) \mapsto L(q(x)u^2) \in \widetilde V
$$
is well defined and $\Phi(u) = \Phi(-u) = \Phi(|u|)$. Moreover, for every $(u, \varphi) \in H_0^1(\Omega) \times \widetilde V$ we have that $\varphi = \Phi(u)$ if and only if for every $\eta \in \widetilde V$
\begin{equation*}
\int_{\Omega} \Delta \varphi \Delta \eta  dx + \int_{\Omega} \nabla \varphi \nabla \eta  dx= \int_{\Omega} q(x)u^2 \eta  dx.
\end{equation*}
By choosing in particular $\eta = \Phi(u)$ we have
\begin{equation}
\label{eq:identity_2}
\int_{\Omega} (\Delta \Phi(u))^2  dx + \int_{\Omega} |\nabla \Phi(u)|^2  dx =  \int_{\Omega} q(x)u^2 \Phi(u)dx.
\end{equation}
Consequently,
\begin{align*}
\|\Phi(u)\|_{\widetilde V}^2 & \leq |q|_\infty \int_{\Omega} u^2 \Phi(u)dx
 \leq c |u|_4^2 |\Phi(u)|_2
 \leq c |\nabla u|_2^2 |\nabla \Phi(u)|_2
 \leq c |\nabla u|_2^2 \|\Phi(u)\|_{\widetilde V}
\end{align*}
so that
\begin{equation}
\label{eq:boundedness_of_Phi}
\|\Phi(u)\|_{\widetilde V} \leq c |\nabla u|_2^2.
\end{equation}
This means that  $\Phi$ is a bounded nonlinear map, namely it sends  bounded sets on bounded sets. A further property is given here.

\begin{lemma}
\label{lem:compactness_of_Phi}
If $u_n \rightharpoonup u$ in $H_0^1(\Omega)$, then
$$
\int_{\Omega} q(x) u_n^2 \Phi(u_n) dx \to \int_{\Omega} q(x) u^2 \Phi(u) dx.
$$
Moreover, the map $\Phi$ is compact.
\end{lemma}

\begin{proof}
Let $u_n \rightharpoonup u$ in $H_0^1(\Omega)$ and define the operators $B_n, B: \widetilde V \longrightarrow \mathbb R$ by
$$
B_n(\eta) := = \int_{\Omega} q(x)u_n^2\eta  dx, \qquad B(\eta) = := \int_{\Omega} q(x)u^2\eta  dx.
$$
In virtue of the H\"older inequality it is easily seen they are continuous; in fact
\begin{align*}
\envert{\int_{\Omega} q(x) u^2 \eta  dx} &
\leq |q|_\infty |u|_4^2 |\eta|_2 \leq c |\nabla \eta|_2 \leq c \|\eta\|_{\widetilde V}
\end{align*}
where  $c$ just depends on $u$. Analogously it can be verified the continuity of $B_{n}$.

Due to the compact embedding of $H^{1}_{0}(\Omega)$ into $L^{p}(\Omega)$ for $p\in[1,6)$, we get $u_{n}^{2} \to u^{2}$ in $L^{6/5}(\Omega)$. Recall that $H_0^1 (\Omega) \hookrightarrow L^r(\Omega)$ for $r \in [1, 6]$, with compact embedding if $r < 6$. Hence $u_n \rightharpoonup u$ in $L^r(\Omega)$ for all $r \in [1, 6]$, hence $\{u_n\}$, and consequently $\{u_n+u\}$ are bounded. So $u_n^2 \to u^2$ in $L^{6/5} (\Omega)$. Indeed:
\begin{align*}
\int_{\Omega} (u_n^2 - u^2)^{6/5}  dx & = \int_{\Omega} (u_n - u)^{6/5} (u_n + u)^{6/5}  dx \\
 & \leq |(u_n - u)^{6/5}|_2 |(u_n+u)^{6/5}|_2 \\
 & \leq c \int_{\Omega} (u_n - u)^{12/5}  dx \to 0.
\end{align*}
Then
\begin{align*}
|B_n(\eta) - B(\eta)| & \leq |q|_\infty |u_n^2-u^2|_{6/5} |\eta|_6
 \leq c |q|_\infty |u_n^2-u^2|_{6/5} \|\eta\|_{\widetilde V}.
\end{align*}
This implies that
$$
\|B_n - B\| \leq \sup_{\eta \neq 0} \frac{c |u_n^2 - u^2|_{6/5} \|\eta\|_{\widetilde V}}{\|\eta\|_{\widetilde V}} \to 0,
$$
namely $B_n \to B$ as operators in  $\widetilde V$.

On the other hand, we have that $\Phi(u_n) \rightharpoonup \Phi(u)$ in $\widetilde V$. Indeed, let $g \in \widetilde{ V}^{'}$. Then there is some $v_g \in \widetilde V$ such that
$$
( g, \Phi(u_n) )  = \int_{\Omega} \nabla \Phi(u_n)  \nabla v_g  dx + \int_{\Omega} \Delta \Phi(u) \Delta v_g  dx = \int_{\Omega} q(x) u_n^2 v_g  dx.
$$
But then
\begin{align*}
(g,\Phi(u_n)) - (g,\Phi(u)) & = \int_{\Omega} q(x) (u_n^2 - u^2) v_g  dx
 \leq |q|_\infty |u_n^2 - u^2|_2 |v_g|_2 \to 0,
\end{align*}
since $u_n^2 \to u^2$ in $L^2(\Omega)$ as well. We then conclude that
$$
\int_{\Omega} q(x) u_n^2 \Phi(u_n) dx \to \int_{\Omega} q(x) u^2 \Phi(u) dx
$$
and by  \eqref{eq:identity_2} that $\| \Phi(u_{n})\|_{\widetilde V} \to \|\Phi(u)\|_{\widetilde V}$. As a consequence  $\Phi(u_n) \to \Phi(u)$ in $\widetilde V$.
\end{proof}

Note that for every $u \in N$ we have that $F_\varphi'(u, \Phi(u)) = 0$. Indeed, $\Phi(u)$ is the unique solution to the problem in Proposition \ref{prop:auxiiary_problem_to_extend_Phi} with $w = q(x)u^2$.

At this point we are in position to  define the reduced functional of a single variable:
$$
J:H^{1}_{0}(\Omega) \to \mathbb R, \quad u\mapsto F(u, \Phi(u)).
$$
With the notation $\varphi_u  = \Phi(u)$ the functionl $J$ is explicitly given by (recall \eqref{eq:identity_2})
\begin{align*}
J(u) & = \frac12 \int_{\Omega} |\nabla u|^2  dx + \frac12 \int_{\Omega} q(x) (\varphi_u +\chi)u^2  dx - \frac{ 1}{p}
\int_{\Omega}
|u|^p  dx \\
& \quad - \frac{1}{4} \int_{\Omega} (\Delta \varphi_u)^2  dx - \frac{1}{4} \int_{\Omega} |\nabla \varphi_u|^2  dx - \frac{\alpha}{2
|\Omega|}\int_{\Omega}
\varphi_u  dx \\
& =\frac12 \int_{\Omega} |\nabla u|^2  dx
+ \frac{1}{4}\int_{\Omega}
q(x) (\varphi_{u} + \chi) u^2  dx - \frac{ 1}{p} \int_{\Omega} |u|^p  dx.
\end{align*}

The functional  $J$ is of class $C^1$ on $H_0^1 (\Omega)$ and even. Moreover, for every $u \in M$ we have that
$$
J'(u)[v] = F_u'(u, \varphi_u)[v] + F_\varphi'(u, \varphi_u)[\Phi'(u)[v]] = F_u'(u, \varphi_u)[v] \quad \forall v \in H_0^1(\Omega)
$$
and hence we deduce the following variational principle.

\begin{theorem}
The pair $(u, \varphi) \in M \times \widetilde V$ is a critical point of $F$ constrained on $M \times \widetilde V$ if and only if $u$ is a critical point of
$J|_M$
and $\varphi = \Phi(u)$.
\end{theorem}

\subsection{Proof of the main result} \label{sec:final}

We recall the following useful result.

\begin{lemma}[\!\!{\cite[Lemma 3.1]{PisaniSiciliano2007}}] \label{lem:ps}
Let $D$ be a regular domain of $\mathbb R^N$ and
\begin{equation*}
\label{eq:lemma_sobolev_spaces_1}
1 \leq s \leq N,
\end{equation*}
\begin{equation*}
\label{eq:lemma_sobolev_spaces_2}
s < p < s^* = \frac{Ns}{N - s}
\end{equation*}
and
\begin{equation*}
\label{eq:lemma_sobolev_spaces_3}
0 < r \leq N \left(1 - \frac{p}{s^*} \right).
\end{equation*}
Then there exists a constant $C > 0$ such that for every $u \in W^{1, s}(D)$ it holds that
\begin{equation*}
\label{eq:lemma_sobolev_spaces_conclusion}
|u|_p^p \leq C \|u\|_{W^{1, s}}^{p - r}|u|_s^r.
\end{equation*}
\end{lemma}

\begin{remark}
If $D$ is bounded, then the conclusion of the lemma is true also in the case $p \in [1, s]$ with $r < p$. Furthermore, if $D$ is bounded and $u \in W_0^{1,p} (D)$, then, by Poincar\'e inequality,
$$
|u|_p^p \leq C |\nabla u|_s^{p - r} |u|_s^r.
$$
\end{remark}

The following lemma gives the existence of solutions to our modified problem.

\begin{lemma}\label{lem:minimo}
\label{lem:coercivity_and_weak_lower_semicontinuity}
The functional $J$ on $M$ is weakly lower semicontinuous and coercive. In particular, it has a  minimizer $u \in M$
and it can be assumed positive.
\end{lemma}

\begin{proof}
We have
\begin{align*}
J(u) & = \frac12 \int_{\Omega} |\nabla u|^2  dx + \frac14 \int_{\Omega} (\Delta \varphi_u)^2 dx + \frac14 \int_{\Omega} |\nabla \varphi_u|^2 dx
 + \int_{\Omega} q(x) \chi u^2  dx - \frac{ 1}{p} \int_{\Omega} |u|^p dx\\
& \geq  \frac12 \int_{\Omega} |\nabla u|^2  dx - |q\chi|_{\infty}
- \frac{ 1}{p} \int_{\Omega} |u|^p dx.
\end{align*}
Finally, we apply Lemma \ref{lem:ps} with $s = 2$ and $N = 3$. Since $p \in (2, 10/3)$ it holds that
$$
p - 2 < 3 \left(1 - \frac p6 \right) < 2
$$
and we can choose
$$
p - 2 < r < 3 \left(1 - \frac p6 \right),
$$
so that by  Lemma \ref{lem:ps} it follows that
$$
\frac{1 }{p} \int_{\Omega}|u|^p  dx \leq c|\nabla u|_2^{p - r}.
$$
Hence,
$$
J(u) \geq \frac12 \int_{\Omega} |\nabla u|^2  dx - |q\chi|_{\infty}
- c'|\nabla u|_2^{p - r}
$$
and thus $J$ is coercive and bounded from below on $M$.

Now, let $\{u_n\} \subset M$ such that $u_{n} \rightharpoonup u$. Since $M$ is weakly closed, $u \in M$. By Lemma~\ref{lem:compactness_of_Phi} we know that
$$
\frac14 \int_{\Omega} (\Delta \varphi_{u_n})^2  dx + \frac14 \int_{\Omega} |\nabla \varphi_{u_n}|^2  dx \to \frac14 \int_{\Omega} (\Delta \varphi_u)^2
dx + \frac14 \int_{\Omega} |\nabla \varphi_u|^2  dx.
$$
We also know that $u_n^2 \to u^2$ in $L^{6/5}(\Omega)$ so
\begin{align*}
\int_{\Omega} q(x) \chi (u_n^2 - u^2)  dx & \leq c \int_{\Omega} |u_n^2 - u^2 | dx
 \leq c |u_n - u|_{6/5} \to 0.
\end{align*}
Finally, the first and last terms are the norms of $u$ in $H_0^1(\Omega)$ and $L^p (\Omega)$ (up to constants), so they are weakly lower semicontinuous.

Thus $J$ is weakly lower semicontinuous and the  existence of the minimum follows by standard results. Note that $J(u) = J(|u|)$, the  minimum can be assumed to be positive.
\end{proof}

For the existence of other critical points beside the minimum, we need the following.

\begin{proposition}
\label{prop:Palais_Smale}
The functional $J$ satisfies the Palais-Smale condition on $M$.
\end{proposition}

\begin{proof}
Let $\{u_n\} \subset M$ be a sequence such that
$$
\{J(u_n)\} \ \text{ is bounded }
$$
and
\begin{equation}
\label{eq:derivative_goes_to_zero}
J|_M'(u_n) \to 0.
\end{equation}
By \eqref{eq:derivative_goes_to_zero}  there exist two sequences of real numbers $\{\lambda_n\}, \{\beta_n \} $
and a sequence $\{v_{n}\} \subset H^{-1}$ such that $v_n \to 0$ and
\begin{equation}
\label{eq:derivative_goes_to_zero_meaning}
- \Delta u_n + q(x)(\varphi_n + \chi)u_n -   |u_n|^{p - 2}u_n = \lambda_n u_n + \beta_n q(x) u_n + v_n
\end{equation}
where $\varphi_n := \varphi_{u_n}$.

Since $J$ is coercive and $\{J(u_n)\}$ is bounded, we deduce that $\{u_n\}$ is bounded in $H_0^1 (\Omega)$. Hence there exists $u \in H_0^1(\Omega)$ such that $u_n \rightharpoonup u$, up to a subsequence. By the compact embeddings and Lemma \ref{lem:compactness_of_Phi} we know that
\begin{equation}\label{eq:conv}
u_{n}\to u \quad\text{in } \  L^{p}(\Omega), \quad
\varphi_n \to \varphi_{u}\quad  \text{ in } \  H^{2}(\Omega).
\end{equation}
Moreover by the compact embedding of $H^{1}_{0}(\Omega)$ into $L^{2}(\Omega)$, $M$ is weakly closed, and then $u \in M$. It only remains to show that $u_n \to u$ in $H_0^1 (\Omega)$.

By \eqref{eq:derivative_goes_to_zero_meaning} we have that
\begin{equation}
\label{eq:derivative_goes_to_zero_applied_to_un}
\frac12 \int_{\Omega} |\nabla u_n|^2  dx + \frac12 \int_{\Omega} q(x)(\varphi_n + \chi)u_n^2  dx - \frac{1 }{p} \int_{\Omega} |u_n|^p  dx - \langle v_n, u_n
\rangle
= \lambda_n + \alpha \beta_n.
\end{equation}
Now the left hand side above is bounded. We just check the second term: by \eqref{eq:conv} we infer
\begin{align*}
    \left| \int_{\Omega} \left(q(x)(\varphi_n + \chi)u_n^2 -  q(\varphi_u + \chi)u^2  \right) dx\right|
    &\le c \int_{\Omega} \left| \varphi_{n}+\chi \right| |u_{n}^{2} - u^{2}|  dx
        +\int_{\Omega} u^{2} \left| \varphi_{n} - \varphi \right|  dx \\
    &= o_{n}(1),
\end{align*}
where we are denoting with $o_{n}(1)$ a vanishing sequence. Then the right-hand side of \eqref{eq:derivative_goes_to_zero_applied_to_un} is bounded and we can assume that
$$
\lambda_n + \alpha \beta_n =\xi + o_{n}(1)
$$
with $\xi \in \mathbb R$.

Then \eqref{eq:derivative_goes_to_zero_meaning} becomes
\begin{equation}
\label{eq:derivative_goes_to_zero_becomes}
-\Delta u_n + q(x)(\varphi_n + \chi)u_n -   |u_n|^{p - 2}u_n - v_n =(\xi + o_{n}(1))u_n - \beta_n(q(x) - \alpha)u_n.
\end{equation}

Now, since $u \in M$ we know that $|u|_2 = 1$. This, together with the assumption $|q^{-1}(\alpha)| = 0$ implies that $(q(x) - \alpha)u$ is not identically zero. Then we can choose a test function $w \in C^{\infty}_{0}(\Omega)$ such that
$$
\int_{\Omega} (q(x) - \alpha)u w  dx \neq 0.
$$
Multiplying \eqref{eq:derivative_goes_to_zero_becomes} by $w$ and integrating  by parts we get
\begin{multline}\label{eq:beta}
\int_{\Omega} \nabla u_n  \nabla w  dx + \int_{\Omega} q(x)(\varphi_n + \chi)u_n w  dx -  \int_{\Omega}|u_n|^{p - 2}u_n w  dx \\ - \langle v_n, w \rangle -
(\lambda + o_{n}(1))\int_{\Omega} u_n w  dx
= \beta_n \int_{\Omega} (q(x) - \alpha)u_n w  dx
\end{multline}
and using again \eqref{eq:conv} we see that every term in the left-hand side converges. Also, by the weak convergence of $\{u_n\}$,
$$
\int_{\Omega} (q(x) - \alpha)u_n w  dx \to \int_{\Omega} (q(x) - \alpha)u w  dx.
$$
This implies, coming back to \eqref{eq:beta}, that $\{\beta_n\}$ is bounded, which in turn implies that $\{\lambda_n\}$ is bounded as well.

Applying \eqref{eq:derivative_goes_to_zero_becomes} to $u_n - u$ we get
\begin{multline}\label{eq:ultimaconv}
\int_{\Omega} \nabla u_n \nabla(u_n - u)  dx + \int_{\Omega} q(x)(\varphi_n + \chi)u_n(u_n - u)  dx -   \int_{\Omega} |u_n|^{p - 2}u_n (u_n - u)dx\\
- \langle v_n, u_n - u \rangle = (\lambda + o(1))\int_{\Omega} u_n (u_n - u)  dx + \beta_n \int_{\Omega} (q(x) - \alpha)u_n(u_n - u)  dx.
\end{multline}
Since (again by \eqref{eq:conv}) we have
\begin{align*}
& \int_{\Omega} q(x)(\varphi_n + \chi)u_n(u_n - u)  dx \to 0, \quad \langle v_n, u_n - u \rangle \to 0, \\
& \int_{\Omega}|u_n|^{p - 2} u_n(u_n - u)  dx \to 0, \quad (\lambda + o(1))\int_{\Omega} u_n (u_n - u)  dx \to 0
\end{align*}
and
$$
\beta_n \int_{\Omega} (q(x) - \alpha)u_n(u_n - u)  dx \to 0,
$$
we conclude by \eqref{eq:ultimaconv} that $|\nabla u_n|_2 \to |\nabla u|_2$ and so $u_n \to u$ in $H_0^1(\Omega)$.
\end{proof}

Now we can give the proof of Theorem \ref{thm:main_theorem}. By Proposition \ref{thm:existence_of_sets_of_genus_k}, $M$ has compact, symmetric subsets of genus $k$ for every $k \in \mathbb N$. Let $n \in \mathbb N$. By Lemma \ref{lem:sublevels_have_finite_genus} there exists some $k \in \mathbb N$ depending on $n$ such that
$$
\gamma(J^n) = k.
$$
Let
$$
A_{k+1} = \left\{A \subset M \ : \ A = -A, \overline A = A, \gamma(A) = k + 1 \right\}
$$
that we know is not empty by Proposition \ref{thm:existence_of_sets_of_genus_k}.

By the monotonicity property of the genus, any $A \in A_{k+1}$ is not contained in $J^n$, then $\sup_A J > n$ and therefore
$$
c_n = \inf_{A \in A_{k+1}} \sup_{u \in A} J(u) \geq n.
$$
It is well known (see e.g., Szulkin \cite{Szulkin1988}) that $c_{n}$ are critical levels for $J|_{M}$ and then there is a sequence $\{u_{n}\}$ of critical points such that
\begin{equation}\label{eq:Jdiverge}
J(u_{n})=\frac12  |\nabla u_{n}|_{2}^2
+ \frac{1}{4}\int_{\Omega}
q(x) (\varphi_{n} + \chi) u_{n}^2  dx - \frac{ 1}{p}  |u_{n}|_{p}^p  dx=
c_{n}\to +\infty.
\end{equation}

The critical points give rise to Lagrange multipliers $\omega_n, \mu_n$ and then, recalling the decomposition $\varphi = \phi - \chi - \mu$, to solutions $(u_n, \omega_n, \phi_n) \in H_0^1(\Omega) \times \mathbb R \times H^2 (\Omega)$ of the original problem.

We show that it is also $\| u_n\|\to + \infty$, meaning that the solutions are more and more ``oscillating'' in some sense. Since
$$
\int_{\Omega} q(x) \chi u_n^2  dx \leq |q \chi|_\infty,
$$
and by \eqref{eq:boundedness_of_Phi} it is
$$
\|\varphi_n\|_{\widetilde V} = \int_{\Omega} (\Delta \varphi_n)^2  dx + \int_{\Omega} |\nabla \varphi_n|^2  dx \leq c |\nabla u_{n}|_2^2,
$$
we infer that
$$
|J(u_n)| \leq (1+c) |\nabla u_n|_{2}^2 + c' |\nabla u_n|_2^{p } + |q \chi|_\infty
$$
and then $\{u_{n}\}$ can not be  bounded in $H^{1}_{0}(\Omega)$.

Note also that  from $J'(u_{n})= \omega_{n} u_{n} - \mu_{n} q(x)u_{n}$ we get
\begin{equation}
|\nabla u_{n}|_{2}^{2} +\int_{\Omega} q(x)(\varphi_{n} + \chi)u_{n}^{2}dx -|u_{n}|_{p}^{p} = \omega_{n}-\mu_{n}\alpha,
\end{equation}
that joint with \eqref{eq:Jdiverge} gives
\[
	+\infty\leftarrow c_{n}<|\nabla u_{n}|_{2}^{2} +\int_{\Omega}q(x)(\varphi_{n}+ \chi)u_{n}^{2}dx -|u_{n}|_{p}^{p}=\omega_{n}-\mu_{n}\alpha.
\]
In this way  the proof of Theorem \ref{thm:main_theorem} is concluded.

\subsection*{Acknowledgment}

G. Siciliano was supported by Capes, CNPq, FAPDF Edital 04/2021 - Demanda Espont\^anea, Fapesp grants no. 2022/16407-1, 2022/16097-2, CEPID 24/00923-6 (Brazil), PRIN PNRR, P2022YFAJH ``Linear and Nonlinear PDEs: New directions and applications'', and INdAM-GNAMPA project n. E5324001950001 ``Critical and limiting phenomena in nonlinear elliptic systems'' (Italy). The author would like to thank the referee for the careful reading of the paper and for the suggestions that certainly improved the reading of the paper.

\EditInfo{April 15, 2025}{September 24, 2025}{Jaqueline Godoy Mesquita, Mariel Sáez Trumper, Rafael Potrie and Tiago Macedo}

\end{document}